\documentclass[letterpaper, 10 pt, conference]{ieeeconf}  

\usepackage{amsmath,amssymb,amsfonts}
\usepackage{algorithm}
\usepackage{graphicx}
\usepackage{textcomp}
\usepackage{url}
\usepackage{color}
\usepackage{bm}
\usepackage{cases}
\usepackage{lipsum}
\usepackage{epstopdf}

\usepackage{amsopn}

\usepackage{amsthm}

\usepackage{enumerate}
\usepackage{cancel}
\usepackage{mathrsfs}
\usepackage{mathdots}
\usepackage{euscript}
\usepackage{amscd}
\usepackage{placeins}
\usepackage[T1]{fontenc}
\usepackage{booktabs}

\usepackage{lineno}
\modulolinenumbers[5]

\usepackage[font=small,labelfont=bf]{caption}
\usepackage{subcaption}

\usepackage{algpseudocode}

\usepackage{multirow}
\usepackage{tikz}
\usetikzlibrary{arrows}
\usepackage{verbatim}

\newtheorem{theorem}{Theorem}

\newtheorem{lemma}{Lemma}
\newtheorem{proposition}{Proposition}

\theoremstyle{definition}
\newtheorem{definition}{Definition}

\theoremstyle{remark}
\newtheorem{remark}{Remark}

\newcommand{\nn}{\nonumber}

\newcommand{\bmat}{\left[ \begin{matrix}}
	\newcommand{\emat}{\end{matrix} \right]}
\newcommand{\innerprod}[2]{\langle{#1},\,{#2}\rangle}

\DeclareMathOperator{\prox}{prox}

\DeclareMathOperator{\argmin}{argmin}
\DeclareMathOperator{\trace}{tr}

\DeclareMathOperator{\E}{{\mathbb E}}
\newcommand{\Rbb}{\mathbb R}

\newcommand{\Zbb}{\mathbb Z}
\newcommand{\Sbb}{\mathbb S}

\newcommand{\Nbb}{\mathbb N}
\newcommand{\xb}{\mathbf  x}
\newcommand{\yb}{\mathbf  y}
\newcommand{\sbf}{\mathbf  s}  

\newcommand{\wb}{\mathbf  w}
\newcommand{\vb}{\mathbf  v}

\newcommand{\db}{\mathbf  d}

\newcommand{\ub}{\mathbf  u}

\newcommand{\lb}{\boldsymbol{\ell}}

\newcommand{\zerob}{\mathbf 0}

\newcommand{\Ab}{\mathbf A}

\newcommand{\Eb}{\mathbf E}

\newcommand{\Ib}{\mathbf I}

\newcommand{\Lb}{\mathbf L}

\newcommand{\Sb}{\mathbf S}

\newcommand{\Ub}{\mathbf U}
\newcommand{\Vb}{\mathbf V}

\newcommand{\Gammab}{\boldsymbol{\Gamma}}
\newcommand{\Sigmab}{\boldsymbol{\Sigma}}

\newcommand{\supp}[1]{\mathrm{supp}(#1)}
\renewcommand{\vec}{\mathrm{vec}}

\newcommand{\Dscr}{\mathscr{D}}

\newcommand{\Ncal}{\mathcal{N}}

\newcommand{\F}{\mathrm{F}}

\newcommand{\srm}{\mathrm{s}}

\IEEEoverridecommandlockouts                              

\title{\LARGE \bf
	A Newton Interior-Point Method for $\ell_0$ Factor Analysis}

\author{Linyang Wang, Wanquan Liu, and Bin Zhu
	\thanks{This work was supported in part by Shenzhen Science and Technology Program (Grant No.~202206193000001-20220817184157001), the Fundamental Research Funds for the Central Universities, and the ``Hundred-Talent Program'' of Sun Yat-sen University.}
	\thanks{The authors are with the School of Intelligent Systems Engineering, Sun Yat-sen University, Gongchang Road 66, 518107 Shenzhen, China. 
		Emails: {\tt\small wangly227@mail2.sysu.edu.cn} (L. Wang),
		{\tt\small \{liuwq63, zhub26\}@mail.sysu.edu.cn} (W. Liu and B. Zhu).
	}
}

\begin{document}
\maketitle
\thispagestyle{empty}
\pagestyle{empty}

\begin{abstract}
		Factor Analysis is an effective way of dimensionality reduction achieved by revealing the low-rank plus sparse structure of the data covariance matrix.
        The corresponding model identification task is often formulated as an optimization problem with suitable regularizations.        
        In particular, we use the nonconvex discontinuous $\ell_0$ norm in order to induce the sparsity of the covariance matrix of the idiosyncratic noise. 
        This paper shows that such a challenging optimization problem can be approached via an interior-point method with inner-loop Newton iterations. 
        To this end, we first characterize the solutions to the unconstrained $\ell_0$ regularized optimization problem through the $\ell_0$ proximal operator, and demonstrate that local optimality is equivalent to the solution of a stationary-point equation.
        The latter equation can then be solved using standard Newton's method, and the procedure is integrated into an interior-point algorithm so that inequality constraints of positive semidefiniteness can be handled.
        Finally, numerical examples validate the effectiveness of our algorithm.
\end{abstract}

\section{Introduction}

Factor Analysis (FA) is one of the most widely used tools in multivariate analysis, which has been applied in various fields such as systems and control \cite{bottegal2014modeling}, econometrics \cite{forni2001generalized}, biology \cite{hochreiter2006new}, psychology \cite{shapiro1982rank}, and finance \cite{stock2002forecasting}. 
The standard factor model is particularly simple, and it reads as
\begin{equation}\label{generate_model}
	\yb_i = \Gammab \ub_i + \wb_i,\quad i=1,2,\dots,N,
\end{equation}
where $\yb_i\in\Rbb^p$ is the observed realization of a zero-mean random process indexed by $i\in\Zbb$, $\Gammab\in\Rbb^{p\times r}$ is a loading matrix with full column rank, components of the random vector $\ub_i\sim \Ncal(0,\Ib_r)$ represent the hidden factors, and  $\wb_i\sim \Ncal(0, \hat \Sb)$ is the idiosyncratic noise which is independent of $\ub_i$. The noise covariance matrix is not known and is to be estimated.
Typically it is assumed that $r\ll p$ which means that the observations are explained by a small number of common factors, and thus dimensionality reduction can be achieved.
Under the model assumptions above, the covariance matrix of $\yb_i$ can be expressed as
\begin{equation}\label{Sigma_decomp}
	\hat\Sigmab := \E(\yb_i\yb_i^\top) =\Gammab\Gammab^\top+\hat\Sb = \hat\Lb + \hat\Sb
\end{equation}
where the $p\times p$ matrix $\hat\Lb := \Gammab\Gammab^\top$ has also rank $r$.
Given $N$ i.i.d.~samples $\yb_i$ from the model, the FA problem is usually posed as
estimating the covariance matrices $\hat \Lb$ and $\hat{\Sb}$. Clearly, the factor loading matrix $\Gammab$ can be recovered from $\hat{\Lb}$ despite in a non-unique fashion.

Recent interests in FA center around optimization formulations which aim to find the additive decomposition \eqref{Sigma_decomp} from a noisy estimate $\check{\Sigma}$ of the data covariance matrix $\hat{\Sigma}$.
 see e.g., \cite{bertsimas2017certifiably,ciccone2018factor,FA_TAC_24}.
In particular, the paper \cite{ciccone2018factor} adopted a classic additional assumption that $\hat{\Sb}$ is diagonal which implies that the components in the noise vector $\wb_i$ are independent. On the other hand, the paper \cite{FA_TAC_24} relaxed such an assumption by requiring $\hat{\Sb}$ to be \emph{sparse} as measured by the $\ell_0$ norm which is equal to the number of nonzero elements in a vector or matrix.
In other words, in the latter case the noise components are allowed to be correlated with a small number of other components which should result in a more general model.
Following such an idea, \eqref{Sigma_decomp} can be understood as a low-rank plus sparse decomposition for covariance matrices.

Although the $\ell_0$ norm represents the most direct indicator of the sparsity of variables, the resulting optimization problem is \emph{nonconvex and nonsmooth}, which creates difficulties for both theoretical analysis and algorithmic development.
To deal with this class of problems, the dominant trend in the literature has been utilizing convex relaxation (the $\ell_1$ norm) or other nonconvex (but continuous) regularizations instead of the $\ell_0$ norm itself, see e.g., \cite{agarwal2012noisy,wen2019robust}.
In the past decade, however, there have been works in the optimization community that address the  $\ell_0$ optimization problem directly, see e.g., \cite{zhou2021newton,zhou2021global,zhou2021quadratic}. The key technical tool is the proximal operator with respect to the $\ell_0$ norm which was discovered much earlier \cite{blumensath2008iterative,blumensath2009iterative}.




For the specific $\ell_0$ regularized optimization problem formulated in \cite{FA_TAC_24} for FA, it was solved by the Alternating Direction Method of Multipliers (ADMM) in \cite{Wang-Zhu-2023-SAM} and by the Block Coordinate Descent (BCD) algorithm in \cite{FA_CDC_24}, both showing superior performances in comparison with the $\ell_1$ convex relaxation. These algorithms are simple to implement and are computationally cheap in each iteration but are often slow in terms of the convergence rate. In particular, the ADMM typically converges sublinearly \cite{boyd2011distributed} (assuming convergence) which means that many iterations are needed before reaching practical convergence. In this paper, we aim to accelerate convergence for the solution of the $\ell_0$ FA problem using the interior-point method (IPM) which handles inequality constraints (for positive semidefniteness) with logarithmic barrier functions. In each inner loop of the IPM, an \emph{unconstrained} $\ell_0$ regularized optimization problem is solved with Newton's method, a powerful second-order algorithm which typically converges \emph{quadratically}.
Numerical examples also demonstrate the improved convergence rate of the proposed IPM in comparison with the ADMM and BCD.

	\emph{Notation:}
In this article, we use bold uppercase letters for matrices and lowercase bold letters for vectors. Given a vector $\xb\in\Rbb^m$, let $\Nbb_m:=\{ 1,2,\cdots, m\}$ be the index set for the $m$ components. We write $\operatorname{supp}(\xb)$, a subset of $\Nbb_m$, for the support set of $\xb$ that consists of indices of the nonzero components of $\xb$.  
For a set $T\subseteq A$, $\bar{T}= A \backslash T$ and $|T|$ denote its complementary set and cardinality.
The symbol $\|\xb\|$ represents the standard Euclidean norm.  The set of symmetric matrices of size $p\times p$ is denoted by $\Sbb^p$.
For a matrix $\Ab\in\Sbb^{p}$, $\Ab\succ 0$ and $\Ab\succeq 0$ mean that $\Ab$ is positive definite and positive semidefinite,  respectively.
Given two matrices $\Ub$ and $\Vb$ in $\Sbb^p$, their inner product 
is defined by $\innerprod{\Ub}{\Vb}=\trace(\Ub \Vb)$, and the induced norm $\|\Ub\|_\F=\sqrt{\innerprod{\Ub}{\Ub}}$ is the Frobenius norm.

\section{Problem Review for $\ell_0$ Factor Analysis}\label{sec:prob}

Our problem formulation is taken from \cite{FA_TAC_24}, and it is briefly reviewed next.
Consider the following $\ell_0$ regularized optimization model with positive semidefinite matricial variables $\Lb$ and $\Sb$ of size $p\times p$: 
\begin{equation}\label{primal_opt_form}
	\underset{(\Lb,\,\Sb)\in\Dscr}{\min}\quad   f(\mathbf{L},\mathbf{S})+ C\Vert \mathbf{S}\Vert _0, 
\end{equation}
where, the feasible set
\begin{equation}\label{primal_feasible_set}
	\Dscr:=\{(\Lb, \Sb) : \Lb\succeq 0,\ \Sb\succeq 0,\ \text{and}\ \Lb+\Sb\succ 0 \},
\end{equation}
and 
\begin{equation}\label{f_smooth}
	f(\Lb,\Sb) :=\trace(\mathbf{L})+\mu\left\{ \operatorname{tr}\left[(\mathbf{L}+\mathbf{S}) \check{\mathbf{\Sigma}}^{-1}\right] - \log \operatorname{det}(\mathbf{L}+\mathbf{S})\right\}
\end{equation}
is the smooth part in the objective function. 
The number $C>0$ is a regularization parameter, and
\begin{equation}\label{def_l0_norm}
    \|\Sb\|_0 = \sum_{i=1}^p \sum_{j=1}^p |s_{ij}|_0,
\end{equation}
is the matricial $\ell_0$ norm which counts the number of nonzero entries with $|s_{ij}|_0=1$ if  $s_{ij}\neq 0$, and $|0|_0=0$.
The specific terms in \eqref{f_smooth} call for further explanations: $\trace(\mathbf{L})$ is a convex surrogate for the rank function of a positive semidefinite matrix, the number $\mu>0$ is another regularization parameter, and the part in the brace can be interpreted as model mismatch which is essentially (up to a constant) the \emph{Kullback--Leibler divergence} (see e.g., \cite{LP15}) between the candidate covariance matrix $\Sigmab:=\Lb+\Sb$ and the sample covariance matrix $\check{\Sigmab}$. 
The latter object is often computed from the $N$ samples of $\yb_i$ via an average, i.e.,
\begin{equation}\label{sample_cov}
	\check{\Sigmab} = \frac{1}{N} \sum_{i=1}^{N} \yb_i \yb_i^\top,
\end{equation}
and we assume $\check{\Sigmab}\succ 0$ so that it can be inverted in \eqref{f_smooth}.


%

In order to devise a second-order algorithm to solve the optimization problem \eqref{primal_opt_form}, we need to first recall the $\ell_0$ proximal operator and some related notations which are given next.

\emph{$\ell_0$ proximal operator for  matricial variables in $\Sbb^p$:} Given a parameter $\gamma>0$, the proximal operator of $C\|\cdot\|_0$ is defined as
\begin{equation}\label{prox_matrix}
	\begin{aligned}
		\prox_{\gamma C \|\cdot\|_0} (\Sb) & := \underset{\Vb\in\Sbb^{p}}{\argmin}\ C\|\Vb\|_0 + \frac{1}{2\gamma} \|\Vb-\Sb\|_\F^2 \\
		& = \underset{\Vb\in\Sbb^{p}}{\argmin}\ \sum_{i,\,j} \left[ C|v_{ij}|_0 + \frac{1}{2\gamma} (v_{ij}-s_{ij})^2 \right],
	\end{aligned}
\end{equation}
where the second equality follows from \eqref{def_l0_norm} and the fact that the squared Frobenius norm can be decoupled elementwise.
According to \cite{blumensath2008iterative, blumensath2009iterative, attouch2013convergence}, the proximal operator \eqref{prox_matrix}, which is itself an optimization problem, admits a particularly simple analytic solution:
\begin{equation}\label{prox_scalar_sol}
	[\prox_{\gamma C \|\cdot\|_0} (\Sb)]_{ij} = \begin{cases}
		0, & \text{if}\ |s_{ij}| < \sqrt{2\gamma C} \\
		0\ \text{or}\ s_{ij}, & \text{if}\ |s_{ij}| = \sqrt{2\gamma C} \\
		s_{ij}, & \text{if}\ |s_{ij}| > \sqrt{2\gamma C}.
	\end{cases}
\end{equation}
Indeed, the solution to the matricial $\ell_0$ proximal operator is obtained by applying the scalar solution elementwise. We notice from the middle line of \eqref{prox_scalar_sol}  that the value of the proximal operator is not unique defined for $|s_{ij}| = \sqrt{2\gamma C}$. Thus in general, $\prox_{\gamma C \|\cdot\|_0} (\Sb)$ should be understood as a \emph{set-valued} mapping.

\section{Optimality Theory for $\tau$-Minimization}\label{sec:theory}

The idea of the interior-point method (IPM) is to convert the constrained optimization problem \eqref{primal_opt_form} into a sequence of unconstrained problems by introducing a  suitable barrier function.
In this section, we analyze the unconstrained problem with a fixed barrier parameter $\tau>0$.
More precisely, we consider the following problem, which is  referred to as \emph{$\tau$-minimization} also in the title of this section:
\begin{equation}\label{opt_barrier}
	\underset{ \mathbf{L},\, \mathbf{S}}{\operatorname{min}}    \quad f_\tau(\Lb,\Sb)+C\|\mathbf{S}\|_0,
\end{equation}
where $f_\tau(\Lb,\Sb) := f(\Lb,\Sb) -\tau [\log\det(\Lb) + \log\det(\Sb)]$
is again the smooth part in the objective function, but this time, modified by the logarithmic barrier function.
Let 
\begin{equation}\label{strict_feas_set}
	\Dscr_{\srm}=\{(\Lb, \Sb) : \Lb\succ0,\ \Sb\succ 0\}
\end{equation}
denote the domain of definition of $f_\tau$ which is also the strictly feasible set of \eqref{primal_opt_form}.
Apparently, $f_\tau$ takes finite values within $\Dscr_{\srm}$, while it assumes a value of $+\infty$ for boundary points, i.e., $\Lb$ or $\Sb$ is positive semidefinite and singular.
Therefore, the set $\Dscr_{\srm}$ can be viewed as the effective domain of the $\tau$-minimization problem \eqref{opt_barrier}, and the minimizer will remain inside $\Dscr_{\srm}$.


Since our aim is to develop a second-order algorithm for \eqref{opt_barrier}, the matricial variables are inconvenient. Hence in the following, we will \emph{vectorize} the problem by introducing a suitable basis and doing computations with coordinates. Meanwhile, the $\ell_0$ proximal operator \eqref{prox_scalar_sol} remains intact thanks to the decoupling property of the $\ell_0$ norm \eqref{def_l0_norm}.
Specifically, we construct an orthonormal basis $\{\Eb_1, \Eb_2, \cdots, \Eb_{m}\}$ for $\Sbb^p$, where $m=p(p+1)/2$. 
For example, when $p=2$ and $m=3$, we can take $\{\Eb_1, \Eb_2, \Eb_{3}\}$ to be
\begin{equation*}
    \bmat1&0\\0&0\emat, \ \frac{1}{\sqrt{2}}\bmat0&1\\1&0\emat, \ \text{and}\ \bmat0&0\\0&1\emat.
\end{equation*}
Now the matricial variables $\Lb$ and $\Sb$  can be expressed as
	$\Lb = \sum_{i=1}^{m}l_i \Eb_i$ and $\Sb = \sum_{j=1}^{m}s_j \Eb_j$,
where $\lb=\bmat l_1,l_2,\cdots,l_m\emat^\top$ and $\sbf=\bmat s_1,s_2,\cdots,s_m\emat^\top$ are the coordinate vectors.
Therefore, we can rewrite \eqref{opt_barrier} into the equivalent vectorized form
	\begin{equation}\label{opt_vector}
		\underset{\lb,\,\sbf}{\operatorname{min}}    \quad h_\tau(\lb,\sbf)+C\|\mathbf{s}\|_0,
	\end{equation}
	where
		 $h_\tau(\lb,\sbf):= f_\tau \left(\sum_{i=1}^{m}l_i \Eb_i,\sum_{j=1}^{m}s_j \Eb_j\right)$
is a function of the coordinates defined on the set  
			$\Dscr_{\srm}^{\vec}=\{(\lb,\sbf)\in\Rbb^m\times \Rbb^m : (\Lb, \Sb)\in\Dscr_{\srm}\}$.
 For simplicity, hereafter we let
$g_\sbf(\lb,\sbf):= \nabla_\sbf h_{\tau}(\lb,\sbf)$ and $g_{\lb}(\lb,\sbf) := \nabla_{\lb} h_{\tau}(\lb,\sbf)$ represent the gradients of $h_{\tau}(\lb,\sbf)$ with respect to $\sbf$ and $\lb$, respectively, and the argument $(\lb,\sbf)$ is sometimes omitted if it is clear from the context.

Next, for any fixed $\tau>0$, we introduce the definition of a $\gamma$-stationary point.
\begin{definition}[$\gamma$-stationary point of \eqref{opt_vector}]\label{def_gamma_stat_point_vec}
	The pair $(\lb^*,\sbf^*)$ is called a $\gamma$-stationary point of \eqref{opt_vector} if there exists a positive number $\gamma>0$ such that 
	\begin{subequations}\label{stationary_vector}
		\begin{align}
			g_{\lb}^* & :=g_{\lb}(\lb^*, \sbf^*) =\mathbf{0}, \label{stationary_vector_l}\\
			\sbf^* & \in \prox_{\gamma C \|\cdot\|_0}\left(\sbf^*-\gamma g_{\sbf}^* \right),\label{stationary_vector_s}
		\end{align}
	\end{subequations}
    where $g_{\sbf}^* :=g_{\sbf}(\lb^*, \sbf^*) $ in \eqref{stationary_vector_s}.
\end{definition}

\begin{remark}
The $\gamma$-stationary point in the sense of Definition~\ref{def_gamma_stat_point_vec} is just the P-stationary point in \cite{FA_TAC_24} for the unconstrained optimization problem \eqref{opt_vector}. 
They are both generalizations of the KKT point \cite{boyd2004convex} which apply to our $\ell_0$ regularized optimization problem and enjoy some of optimality properties of the usual KKT points for smooth optimization problems.
If the problem were smooth, a $\gamma$-stationary point reduces to the usual stationary point at which the gradient of the objective function vanishes, much like  \eqref{stationary_vector_l}.
\end{remark}

With the help of the $\ell_0$ proximal operator, the $\gamma$-stationary point in Definition~\ref{def_gamma_stat_point_vec} can be characterized via the next lemma whose proof is relatively simple and hence omitted. 
\begin{lemma}\label{lem_equivalent}
	A pair $(\lb^*,\sbf^*)$ is a $\gamma$-stationary point of \eqref{opt_vector} if and only if 
				\begin{equation}\label{stationary_equivalent}
					\begin{cases}
						g_{\lb}^*=\mathbf{0},\\	
						 g_{s_{i}}^* := \frac{\partial}{\partial s_i}h_{\tau}(\lb,\sbf)  = 0  \ \text{and}\ |s_i|\geq\sqrt{2\gamma C}, & i \in \operatorname{supp}(\sbf^*), \\
						| g_{s_i}^*| \leq \sqrt{2 C / \gamma},& i \notin \operatorname{supp}(\sbf^*).
					\end{cases}	
				\end{equation}
\end{lemma}

For the vectorized $\tau$-minimization problem \eqref{opt_vector}, the theoretical analysis of optimality relies on the following two propositions whose proofs are rather standard (see e.g., \cite{FA_TAC_24}) and hence omitted.
\begin{proposition}\label{prop_jointly_convex}
	  The function $h_\tau(\lb,\sbf)$ in \eqref{opt_vector} is jointly strongly convex in $(\lb,\sbf)$.
\end{proposition}
\begin{proposition}\label{prop_strongly_smooth}
	  The function $h_\tau$ in \eqref{opt_vector} is strongly smooth with a positive constant $K$, that is, it satisfies
	  \begin{equation}
	  	\begin{aligned}
		  	   h_{\tau}(\ub, \vb) \leq h_{\tau}(\lb,\sbf)+
		  	    \innerprod{g_{\lb}}{\ub-\lb}+\innerprod{g_{\sbf}}{\vb-\sbf}\\
		  	    +\frac{K}{2}\|(\ub,\vb)-(\lb,\sbf)\|^2
	  	\end{aligned}
	  \end{equation}
      for any $(\lb, \sbf)$ and $(\ub,\vb)$ in a compact subset of $\Dscr_{\srm}^{\vec}$.
\end{proposition}

\begin{theorem}\label{thm_opt}
    We have the following claims.
	\begin{enumerate}
		\item A global minimizer $(\lb^*, \sbf^*)$ of \eqref{opt_vector} is a $\gamma$-stationary point where the parameter $\gamma\in(0, 1/K)$. Moreover, we can replace \eqref{stationary_vector_s} with  an equality
  \begin{equation}\label{prox_s_equality}
    \sbf^*=\prox_{\gamma C \|\cdot\|_0}\left(\sbf^*-\gamma g^*_\sbf \right), 
  \end{equation}
  meaning that the latter proximal operator is single-valued in this case.
		\item Any $\gamma$-stationary point is a local minimizer of \eqref{opt_vector}.      
	\end{enumerate}
\end{theorem}
  \begin{proof}
  	
  1) Let $(\lb^*, \sbf^*)$ be a global minimizer. Then for any $\ub\in\Rbb^m$, we have
$h_\tau(\lb^*,\sbf^*)+C\|\sbf^*\|_0\leq h_\tau(\ub,\sbf^*)+C\|\sbf^*\|_0$, 
which means that $\lb^* = \argmin_{\lb\in\Rbb^m\ \text{s.t.}\ \Lb\succ 0}\ h_\tau(\lb, \sbf^*) $.
  Since the function $h_\tau(\lb,\sbf)$ is smooth with respect to $\lb$, $\lb^*$ must be a stationary point of  $h_\tau(\lb, \sbf^*) $, meaning $g_{\lb}^*=\mathbf{0}$.
  
  Next, for any $\vb\in \prox_{\gamma C \|\cdot\|_0}\left(\sbf^*-\gamma g^*_\sbf \right)$, we have
    \begin{subequations}
  	\begin{align}
  		0 &\leq h_\tau(\lb^*,\vb)+ C\|\vb\|_0- h_\tau(\lb^*,\sbf^*)-C\|\sbf^*\|_0 \label{ineq_1}\\ 
  		&\leq C\|\vb\|_0-C\|\sbf^*\|_0 +\innerprod{g^*_\sbf}{\vb-\sbf^*}+\tfrac{K}{2}\|\vb - \sbf^*\|^2  \label{ineq_2} \\
  		& \leq  \innerprod{g^*_\sbf}{\vb-\sbf^*}+\tfrac{K}{2}\|\vb - \sbf^*\|^2 +\tfrac{\gamma}{2}\|g^*_\sbf\|^2\nn \\ 
  		& \quad -\tfrac{1}{2\gamma}\|\vb-\sbf^*+\gamma g^*_\sbf\|^2 \label{ineq_3}
  		\\
  		&= (\tfrac{K}{2}-\tfrac{1}{2\gamma})\|\vb-\sbf^*\|^2 \leq 0, \label{ineq_4}
  	\end{align}
  \end{subequations}
  where inequalities \eqref{ineq_1}, \eqref{ineq_2}, \eqref{ineq_3} and \eqref{ineq_4} follow from the facts that $(\lb^*,\sbf^*)$ is the global minimizer of \eqref{opt_vector}, $h_\tau$ is strongly smooth, the definition of $\vb$, and $\gamma<1/K$, respectively. Consequently, we derive that $\|\vb-\sbf^*\|=0$ and thus \eqref{prox_s_equality} holds.
  
  2) Suppose that $(\lb^*,\sbf^*)$ is a $\gamma$-stationary point of \eqref{opt_vector}, and denote $T_*:=\operatorname{supp}(\sbf^*)$. We first construct a neighborhood of $(\lb^*,\sbf^*)$ as $U(\lb^*,\sbf^*)=\{(\lb,\sbf)\in \Dscr_{\srm}^{\vec} : \|(\lb,\sbf)-(\lb^*,\sbf^*)\|<\epsilon_*\}$,
  where $\epsilon_* :=\min\{\min_{i\in T_*} |s_i^*|, \sqrt{\gamma C/(2m)}\}$. Then we proceed to show the local optimality of $(\lb^*,\sbf^*)$ in such a neighborhood.
  
  Firstly, for any $(\lb,\sbf)\in U(\lb^*,\sbf^*)$, we assert that $T_*\subseteq \operatorname{supp}(\sbf)$. If not, it would imply the existence of some  $j\in T_*$ but $j\notin \operatorname{supp}(\sbf)$, i.e., $s_j=0$, which leads to a contradiction
  \begin{equation}
  	 \epsilon_* \leq \min _{i \in T_*}\left|s_i^*\right| \leq\left|s_j^*\right|=\left|s_j^*-s_j\right| \leq\left\|\sbf-\sbf^*\right\|<\epsilon_* .
  \end{equation}
  Hence, the assertion $T_*\subseteq \operatorname{supp}(\sbf)$ is true.
  On the basis of the jointly convexity of $h_\tau$ (Proposition~\ref{prop_jointly_convex}) and Lemma \ref{lem_equivalent}, we have
  \begin{equation}\label{inequal_alpha}
  	 \begin{aligned}
  	 	  &h_\tau(\lb,\sbf)-h_\tau(\lb^*,\sbf^*) \\
  	 	  &\geq \innerprod{g^{*}_{\lb}}{\lb-\lb^*}+\innerprod{g^{*}_{\sbf}}{\sbf-\sbf^*} \\
  	 	  &=\innerprod{g_{{\sbf}_{T_*}}^*}{(\sbf-\sbf^*)_{T_{*}}} +\innerprod{g_{{\sbf}_{\bar{T}_*}}^*}{(\sbf-\sbf^*)_{\bar{T}_{*}}} \\
  	 	  &=\innerprod{g_{{\sbf}_{\bar{T}_*}}^*}{(\sbf-\sbf^*)_{\bar{T}_{*}}} =: \alpha.
  	 \end{aligned}
  \end{equation}
  Now we consider two different cases:
        
        \noindent\emph{Case 1}. If $T_*=\operatorname{supp}(\sbf)$, it follows that $\|\sbf\|_0=\|\sbf^*\|_0$ and $\sbf_{\bar{T}_{*}}=\sbf^*_{\bar{T}_{*}}=\zerob$ which together imply that  $\alpha=0$ and
  		\begin{equation}\label{local_optim_case1}
  			\begin{aligned}
  				h_\tau(\lb,\sbf)+C\|\sbf\|_0&\geq h_\tau(\lb^*,\sbf^*)+\alpha+C\|\sbf\|_0 \\
  				&=h_\tau(\lb^*,\sbf^*)+C\|\sbf^*\|_0,
  			\end{aligned}
  		\end{equation}
        where we have used \eqref{inequal_alpha} in the first inequality.
        
        \noindent\emph{Case 2}. If $T_*\subset\operatorname{supp}(\sbf)$ is a strict inclusion, the consequence is that 
        \begin{equation}\label{inequal_l0_norm_s}
            \|\sbf\|_0\geq\|\sbf^*\|_0+1. 
        \end{equation}
        In addition, we have the chain of inequalities
  		\begin{equation}\label{inequal_alpha>-C}
  			\begin{aligned}
  				\alpha \geq-\|g_{{\sbf}_{\bar{T}_*}}^*\| \|(\sbf-\sbf^*)_{\bar{T}_{*}}\| & \geq - \sqrt{2C|\bar{T}_*|/\gamma} \|\sbf_{\bar T_*}-\sbf^*_{\bar{T}_{*}}\| \\
  				&\geq -\sqrt{2Cm/\gamma}\epsilon_*>-C,
  			\end{aligned}
  		\end{equation}
        where,
        \begin{itemize}
            \item the first inequality is Cauchy--Schwarz,
            \item the second inequality is implied by the third line of  \eqref{stationary_equivalent},
            \item the third inequality follows from the definition of the neighborhood $U(\lb^*,\sbf^*)$ and the fact that $|\bar{T}_*|\leq m$,
            \item and the last inequality comes from the definition of  $\epsilon_*$.
        \end{itemize}
  	Consequently, it can be derived that 
  		\begin{equation}\label{local_optim_case2}
  			\begin{aligned}
  				h_\tau(\lb,\sbf)+C\|\sbf\|_0&\geq h_\tau(\lb^*,\sbf^*)+\alpha+C\|\sbf\|_0 \\
  				&>h_\tau(\lb^*,\sbf^*)+C\|\sbf\|_0 -C\\
  				&\geq h_\tau(\lb^*,\sbf^*)+C\|\sbf^*\|_0,
  			\end{aligned}
  		\end{equation}
        where we have used \eqref{inequal_alpha}, \eqref{inequal_alpha>-C}, and \eqref{inequal_l0_norm_s}, respectively, for the three inequalities.
  	  Therefore, the local optimality of the $\gamma$-stationary point $(\lb^*,\sbf^*)$ holds in both cases as we have shown \eqref{local_optim_case1} and \eqref{local_optim_case2}.
\end{proof}
Theorem~\ref{thm_opt} implies that a locally optimal solution of the problem \eqref{opt_vector} can be obtained by finding a $\gamma$-stationary point. 
To provide a representation of the $\gamma$-stationary point that is more amenable to second-order algorithms, we define the index set
  \begin{equation}\label{index_set_T}
	T:=T_\gamma(\lb,\sbf;C)=\left\{i\in\Nbb_m:\left|s_i-\gamma g_{s_i} \right|\geq\sqrt{2\gamma C}\right\},
\end{equation}
and then construct the  stationary-point equation
 \begin{equation}\label{stationary_equations}
	F_\gamma (\lb,\sbf; T) := \left[\begin{array}{c}
		g_{\lb } (\lb,\sbf) \\
		g_{\sbf_{T} }(\lb,\sbf)  \\
		\sbf_{\bar{T}}
	\end{array}\right]=\zerob.
\end{equation}
Notice that by Lemma~\ref{lem_equivalent}, for a $\gamma$-stationary point $(\lb^*,\sbf^*)$, the index set $T_{\gamma}(\lb^*,\sbf^*;C)$ in \eqref{index_set_T} coincides with $T_*=\supp{\sbf^*}$ used in the proof of the second claim of Theorem~\ref{thm_opt}.

The following theorem demonstrates a kind of equivalence between solutions to \eqref{stationary_equations} and $\gamma$-stationary points of \eqref{opt_vector}.
\begin{theorem}\label{thm_F_equation}
	 	We have the following claims. 
	 	\begin{enumerate}
	 		\item If $(\lb^*,\sbf^*)$ is a $\gamma$-stationary point of \eqref{opt_vector} satisfying \eqref{prox_s_equality}, then the 
            equality \eqref{stationary_equations}
            holds for $(\lb^*,\sbf^*)$.
	 		\item  A point $(\lb^*,\sbf^*)$ satisfying \eqref{stationary_equations}
            is also a $\gamma$-stationary points of \eqref{opt_vector}.
	 	\end{enumerate}
\end{theorem}
\begin{proof}
	1)
     For the $\gamma$-stationary point $(\lb^*,\sbf^*)$,
     we define the index set $T_*=T_\gamma(\lb^*,\sbf^*;C)$ with $T_\gamma$ in \eqref{index_set_T}.
     If $(\lb^*,\sbf^*)$ satisfies \eqref{prox_s_equality}, i.e., the set $\prox_{\gamma C \|\cdot\|_0}\left(\sbf^*-\gamma g_{\sbf}^* \right)$ is a singleton, then according to \eqref{prox_scalar_sol}, there is no index $i\in \Nbb_m$ such that $\left|\sbf^*_i-\gamma g^*_{s_i}\right|=\sqrt{2\gamma C}$. 
    Hence, we can deduce
	\begin{equation}
		\begin{aligned}
			\mathbf{0}& = 
			\left[\begin{array}{c}
				g_{\lb}^* \\   	 	  
				\sbf ^*- \prox_{\gamma C \|\cdot\|_0}\left(\sbf^*-\gamma g_{\sbf}^* \right)
			\end{array}\right] \\
			& = \left[\begin{array}{c} \mathbf{0} \\
				\sbf_{T_*}^*
				\\ 
				\sbf_{\bar{T}_*}^*\end{array}
			\right]
            -\left[\begin{array}{c} -g_{\lb}^* \\
				(\sbf^*-\gamma g^*_{\sbf})_{T_*}
				\\ 
				\mathbf{0}\end{array}
			\right] 
			= 
			\left[\begin{array}{c} g_{\lb}^* \\
				\gamma g_{\sbf_{T_*}}^*
				\\ \sbf^*_{\bar{T_*}}
			\end{array}
			\right],
		\end{aligned}
	\end{equation}
	where the first equality comes from Definition~\ref{def_gamma_stat_point_vec},
    and the second equality is derived from 
    \eqref{prox_scalar_sol}.
    
	2) 
	Suppose that $F_\gamma (\lb^*,\sbf^*; T_*)=\mathbf{0}$ with $T_*=T_\gamma(\lb^*,\sbf^*;C)$.
     Then for any $i \in T_*$, we have $g^*_{s_i}=0$ from \eqref{stationary_equations}, and further obtain $\left|\sbf^*_i\right|\geq \sqrt{2\gamma C}$ from \eqref{index_set_T}. Similarly, for any $i \in \bar{T}_*$, we have $\sbf^*_i=0$ and 
	$\left|\gamma g^*_{s_i} \right|\leq\sqrt{2\gamma C}$. Appealing to Lemma~\ref{lem_equivalent}, we can conclude that $(\lb^*,\sbf^*)$ is a $\gamma$-stationary point of \eqref{opt_vector}. 
\end{proof}

  Theorem~\ref{thm_F_equation} provides a useful characterization of $\gamma$-stationary points as solutions to the set of nonlinear equations \eqref{stationary_equations}. When such a characterization is used in conjunction with Theorem~\ref{thm_opt}, we see that a local minimizer of the problem $\eqref{opt_vector}$ can be obtained by solving \eqref{stationary_equations}.

\section{Algorithm Design for the Newton IPM}\label{sec:alg}

Theorem~\ref{thm_opt} has established the connection between local minimizers of the problem \eqref{opt_vector} and $\gamma$-stationary points.
We must notice that these local minimizers are interior points, i.e., they belong to the strict feasible set $\Dscr_{\srm}$ in \eqref{strict_feas_set}, while for the original problem \eqref{primal_opt_form}, we seek solutions on the boundary of $\Dscr_{\srm}$ because $\Lb$ is expected to have low rank. In order to achieve such a boundary solution, the IPM drives a series of solutions to the $\tau$-minimization problem with barrier parameters in $\{\tau_k\}_{k\geq 1}$ such that $\tau_k\to 0$ as $k\to \infty$.
In practice, the barrier parameter $\tau_k$ is updated according to a suitable decreasing rule, e.g., as a geometric sequence, and the algorithm terminates when $\tau$ is less than some specified threshold.
In addition, we employ the \emph{warm start} strategy in the IPM, where the initial point of $\tau_{k+1}$-minimization is the optimal solution of $\tau_k$-minimization, which can effectively improve the convergence speed.
Below we give the framework of the IPM to solve our $\ell_0$ FA problem \eqref{primal_opt_form}.
\begin{algorithm}[H]
	\caption{Interior-point method}
	\label{alg:ipm}
	\begin{algorithmic}[1]
		\Require $C$, $\mu$, $\gamma>0$, $\tau=\tau_0$, $\epsilon>0$, $\theta\in(0, 1)$, $(\Lb^0, \Sb^0)$.
		\State Vectorize $(\Lb^0,\Sb^0)\mapsto(\lb^0,\sbf^0)$.
		\While{$\tau > \epsilon$}
		\State Solve the $\tau$-minimization problem $\eqref{opt_vector}$ for a local minimizer $(\lb^*(\tau),\sbf^*(\tau))$.
		\State  $\tau  = \theta \tau$.
		\EndWhile 
		\Statex\textbf{Output:} The last $(\lb^*(\tau),\sbf^*(\tau))\mapsto(\Lb^*,\Sb^*)$.
	\end{algorithmic}
\end{algorithm}

Next, we only need to focus on solving the $\tau$-minimization problem  \eqref{opt_vector} inside the loop of the IPM for a fixed $\tau>0$. By Theorem~\ref{thm_F_equation}, this is equivalent to solving the stationary-point equation \eqref{stationary_equations} for which we use Newton's method for nonlinear equations.
 
Let $(\lb^k, \sbf^k)$ be the current iterate, and define the index set $T_k:=T_\gamma (\lb^k,\sbf^k;C)$. 
We first compute the Newton direction $\db^k$ which is the solution to the linear system of equations
\begin{equation}
	\nabla F_\gamma (\lb^k, \sbf^k; T_k) \,\db^k = - F_\gamma (\lb^k, \sbf^k; T_k).
\end{equation}
More specifically, we can partition the equations as follows:
\begin{equation}\label{eq_d_k}
	\!
	 \left[\!\!\begin{array}{cc|c}
		\nabla_{\lb, \lb}^2 h_\tau & \nabla_{\lb, \sbf_{T_k}}^2 h_\tau & \nabla_{\lb, \sbf_{\bar{T}_k}}^2 h_\tau \\
		\nabla_{\sbf_{T_k}, \lb}^2 h_\tau & \nabla_{\sbf_{T_k}, \sbf_{T_k}}^2 h_\tau & \nabla_{\sbf_{T_k}, \sbf_{\bar{T}_k}}^2 h_\tau \\
		\hline
		\mathbf{0} & \mathbf{0} & \Ib
	 \end{array}\!\!\right] \!\!\!
			\bmat
				\db_{\lb}^k \\
				\db_{\sbf_{T_k}}^k \\
				\db_{\sbf_{\bar{T}_k}}^k
				\emat
	\!\!=\!-\!
	\bmat
		g_{\lb}^k  \\
		g_{\sbf_{T_k}} ^k \\
		\sbf_{\bar{T}_k}^k
	\emat\!\!.
\end{equation}
It is not difficult to see that $\db_{\sbf_{\bar{T}_k}}^k=-\sbf_{\bar{T}_k}^k$, and by back substitution we find that $\db^k_{\lb\cup\sbf_{T_k}}$ solves the reduced linear equation
\begin{equation}\label{newton_direction}
	\left(\nabla_{\lb\cup\sbf_{T_k}, \lb\cup\sbf_{T_k }}^2 h_\tau\right) 
	\db^k_{\lb\cup\sbf_{T_k}}
	=\bmat
		\left(\nabla_{\lb, \sbf_{\bar{T}_k}}^2 h_\tau\right) \sbf_{\bar{T}_k}^k-	g_{\lb}^k  \\
		\left(\nabla_{\sbf_{T_k}, \sbf_{\bar{T}_k}}^2 h_\tau\right) \sbf_{\bar{T}_k}^k-g_{\sbf_{T_k}} ^k 
	\emat
\end{equation}
of size $m+|T_k|$. Clearly, the reduced coefficient matrix in \eqref{newton_direction} contains the four blocks in the upper-left corner of the full coefficient matrix in \eqref{eq_d_k}. Notice that the full Hessian $\nabla^2 h_\tau$ of $h_\tau$ is positive definite by Proposition~\ref{prop_jointly_convex}. Consequently, the reduced coefficient matrix in \eqref{newton_direction} is always positive definite (and thus invertible) because it is a principal submatrix of $\nabla^2 h_\tau$.

Then we need to determine the stepsize $\alpha$.
If we choose a unit stepsize, then we observe that $\sbf_{\bar T_k}^{k+1}=\sbf_{\bar T_k}^k+\db^k_{\bar T_k}=\zerob$ which indicates $\operatorname{supp}(\sbf^{k+1})\subseteq T_k$.
For this reason, we modify the standard rule $(\lb^{k+1},\sbf^{k+1})=(\lb^k,\sbf^k)+\alpha \db^k$ as $(\lb^{k+1},\sbf^{k+1})=(\lb^k(\alpha),\sbf^k(\alpha))$ where $\lb^k(\alpha)=\lb^k+\alpha \db^k_{\lb}$
and 
\begin{equation}
	\sbf^k(\alpha)=\bmat
		\sbf^k_{T_k}+\alpha \db^k_{\sbf_{T_k}} \\
		\sbf^k_{\bar{T}_k}+\db^k_{\sbf_{\bar{T}_k}} 
	\emat
	=\bmat
		\sbf^k_{T_k}+\alpha \db^k_{\sbf_{T_k}} \\
		\mathbf{0}
	\emat.
\end{equation}
In plain words, we take a unit stepsize for the block of variables $\sbf^k_{\bar{T}_k}$ while do line search to determine a stepsize $\alpha$ for the rest variables.
 Algorithm \ref{alg:newton} summarizes the steps of Newton's method for a solution of \eqref{eq_d_k}.
\begin{algorithm} 
	\caption{Newton's method for (\ref{opt_vector})}
		\label{alg:newton}
	\begin{algorithmic}[1]
		\Statex \textbf{Input:} $C$, $\mu$, $\gamma>0$, $\epsilon>0$, $\delta>0$, $\sigma\in(0, \frac{1}{2})$, $\beta\in(0, 1)$, $(\lb^0,\sbf^0)$.
		\While{the stopping conditions are not satisfied}
		\State Compute $T_k = \{i\in \Nbb_m:|s_i^k-\gamma g_{s_i}^k|\geq \sqrt{2\gamma C}\}.$

        \If {$\db^k$ satisfies
				\begin{equation}
						\innerprod{g^k_{\sbf_{T_k}}}{\db^k_{\sbf_{T_k}} } \leq -\delta\|\db^k_{\sbf}\|^2+\|{\sbf}^k_{\bar{T}_k}\|^2/(4\gamma),\nn
				\end{equation}}
			   \State update $\db^k$ by solving \eqref{eq_d_k}. 
			\Else 
                \State update $\db^k$ by
			\begin{equation*}
				 \db^k_{\lb} =-g^k_{\lb},\  \db^k_{\sbf_{T_k}}=-g^k_{\sbf_{T_k}},\  \db^k_{\sbf_{\bar{T}_k}}=-\sbf_{\bar{T}_k}^k.
			\end{equation*}
            \EndIf
		\State Find the smallest nonnegative integer $v_k$ such that
		\begin{equation*}
			h_\tau(\lb^{k+1}(\beta^{v_k}), \sbf^{k+1}(\beta^{v_k}))\leq h_\tau(\lb^{k},\sbf^k)+\sigma\beta^{v_k}\innerprod{g^k}{\db^k}.
		\end{equation*}
		\State Update $(\lb^{k+1}, \sbf^{k+1})=(\lb^{k}(\alpha^k), \sbf^{k}(\alpha^k))$ with $\alpha^k=\beta^{v_k}$.
		\EndWhile 
		\Statex\textbf{Output:} $(\lb^*, \sbf^*)=$ the last $(\lb^k, \sbf^k)$.
	\end{algorithmic}
\end{algorithm}

\section{Numerical Examples}\label{sec:experiment}
 
 In this section, we show the convergence performance of the IPM via simulations on synthetic data sets. The basic setup of our simulations are given next.
 
 
\emph{Synthetic data description}.
With reference to the factor model \eqref{generate_model},  we take the size of the observation vector $p=40$, and the number of hidden factors $r=5$.
Then we randomly generate a loading matrix $\Gammab\in\Rbb^{p\times r}$ with linearly independent columns and a positive-definite sparse (nondiagonal) matrix $\hat{\Sb}$ such that the signal-to-noise ratio (SNR) equals to $1$, where the SNR is defined as $\|\Gammab\Gammab^\top\|_\F/\|\hat\Sb\|_\F$.
Next we generate a sample sequence $\{\yb_1,\yb_2,\cdots ,\yb_N\}$ of length $N=1200$.




\emph{Initialization of Algorithm~\ref{alg:ipm}}.
Since the effective domain of the problem \eqref{opt_vector} consists of interior points only,  the initial variables $\Lb^0$ and $\Sb^0$ must be positive definite. Therefore, 
 we initialize Algorithm~\ref{alg:ipm} with $(\Lb^0, \Sb^0)=(\frac{1}{2}\check{\Sigmab}, \frac{1}{2}\check{\Sigmab})$, where $\check{\Sigmab}$ is the sample covariance matrix computed by \eqref{sample_cov}.

\emph{Stopping conditions}.
The termination of the outer iteration in Algorithm \ref{alg:ipm} is determined by the parameter $\epsilon$, where a smaller value of $\epsilon$ leads to a better approximate solution. In our simulations, we set $\epsilon=10^{-6}$, and the initial value of $\tau$ as $\tau_0=0.5$. According to Theorems~\ref{thm_opt} and \ref{thm_F_equation}, Algorithm \ref{alg:newton} can be considered to stop when
$$\frac{\|F_\gamma(\lb^k,\sbf^k;T_k)\|}{\sqrt{2m}}\leq 10^{-4},$$
where $m$ is the size of the vectorized variable $\lb$ or $\sbf$.

 \emph{Parameters setting}.
In our implementation of Algorithm~\ref{alg:newton}, we set $ \delta=10^{-4}$, $\sigma = 5\times 10^{-5}$, and $\beta=0.5$.  The parameter $\gamma$ in \eqref{stationary_vector_s}, which can be interpreted as a stepsize in proximal gradient descent, directly affects the sparsity of the optimal $\sbf^*$, and must be adjusted by trial and error for each specific data set, see \cite[Sec.~V]{FA_TAC_24}.
For the choice of the tuning parameters $C$ and $\mu$, we employ  the same Cross Validation procedure as described in \cite{FA_TAC_24}.

\begin{figure}[t]
	\begin{minipage}[b]{.49\linewidth}
		\centering
		\centerline{\includegraphics[width=1\linewidth]{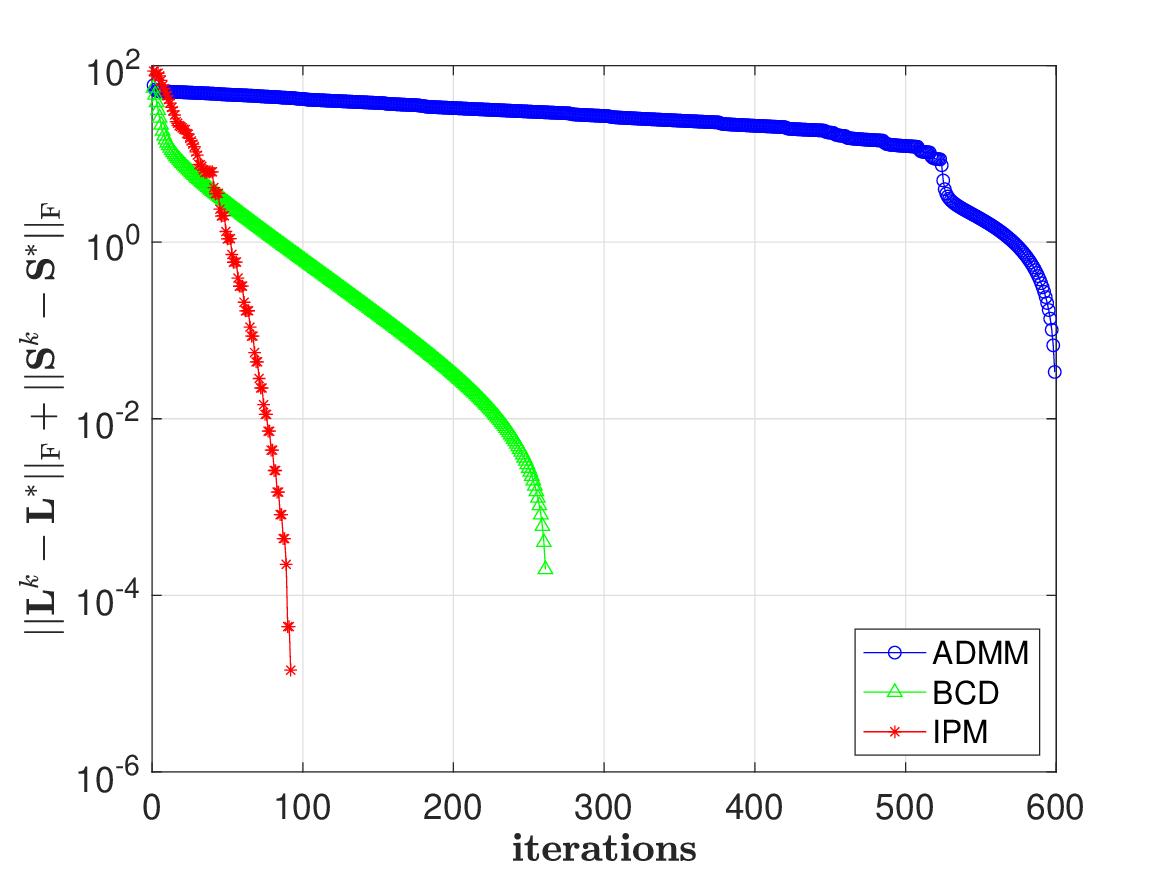}}
		\centerline{(a) $\theta=0.5$}
	\end{minipage}
	\hfill
	\begin{minipage}[b]{.49\linewidth}
		\centering
		\centerline{\includegraphics[width=1\linewidth]{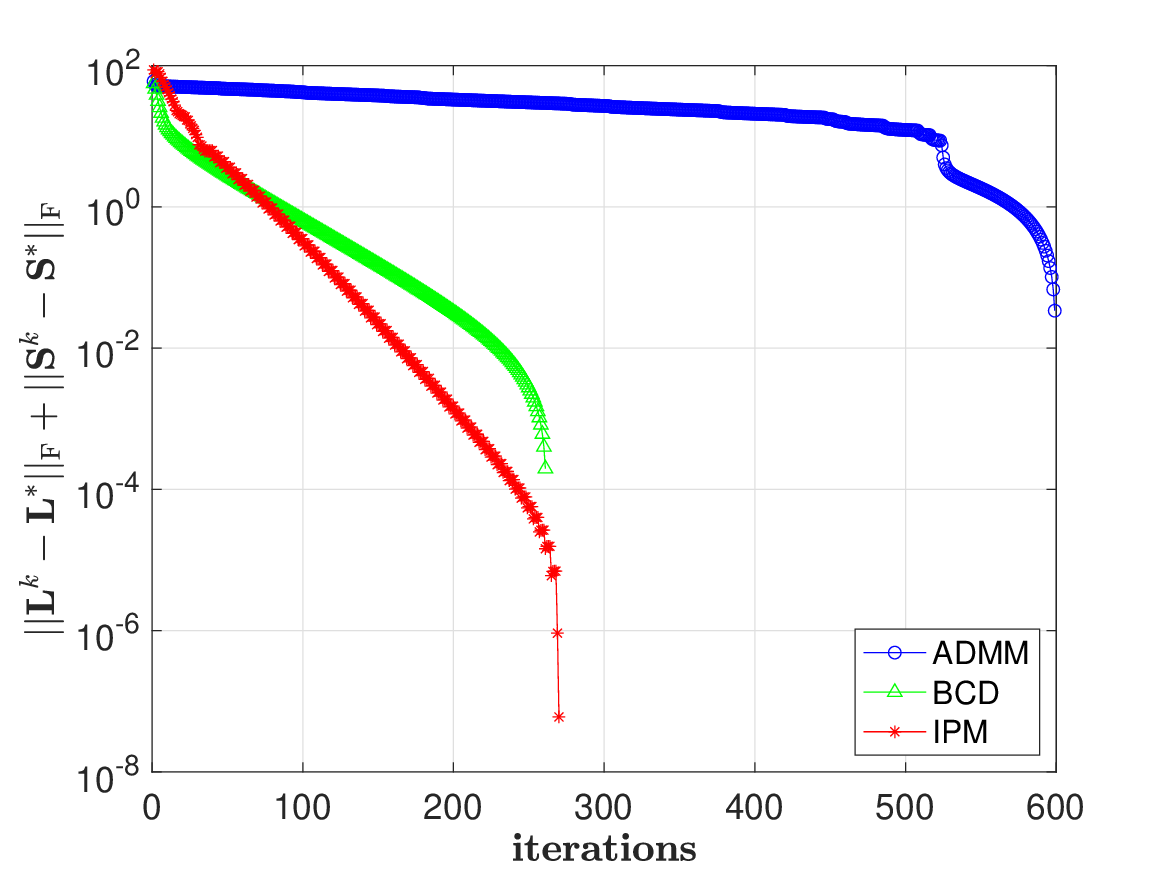}}
		\centerline{(b) $\theta=0.8$}
	\end{minipage}
    \caption{Comparison of the convergence rates among the three algorithms ADMM, BCD, and
    IPM applied to the same data set with different values of the parameter $\theta$ for the IPM.  
    }
\label{fig:convergence}
\end{figure}



\emph{Simulation results}.The number of outer iterations in Algorithm~\ref{fig:convergence} is determined by the decay ratio $\theta$ for the barrier parameter $\tau$.
    On the synthetic data set described at the beginning of this section, we have run the ADMM in \cite{Wang-Zhu-2023-SAM}, BCD in \cite{FA_CDC_24}, and our proposed IPM with two different values of $\theta$, and their convergence rates  are compared in  Fig.~\ref{fig:convergence}.
	The left and right panels correspond to $\theta=0.5$ and $0.8$, respectively.
	It must be noted that in Fig.~\ref{fig:convergence}, the horizontal axis (iterations) corresponding to the red curve does not represent the outer iterations of Algorithm \ref{alg:ipm},
	but rather the total number of inner-loop iterations performed by  Newton's method for all the values of $\tau$ starting from $\tau_0$ until the termination of Algorithm \ref{alg:ipm}.
	During this process, one can work with the geometric progression of $\tau$ and deduce  that   the total number of outer iterations of the IPM is $18$ and $58$ for $\theta=0.5$ and $0.8$, respectively. It has also been observed that Newton's method typically converges within six iterations for each inner loop in both cases starting from the first value of $\tau$ that is less than $10^{-2}$.
	Clearly, compared to the sublinear convergence of the ADMM and the linear convergence of the BCD algorithm, the IPM exhibits a significantly superior convergence rate which appears to be superlinear. It should also be noted, however, that Newton's method is usually computationally expensive, because it involves second-order derivatives. More comprehensive simulation studies 
    will be carried out in a future work.
    



\section{Conclusion}\label{sec:conc}

This article has explored the interior-point method (IPM) for $\ell_0$ factor analysis which reformulates the original inequality constrained optimization problem into a series of unconstrained $\ell_0$ regularized optimization problems using the logarithmic barrier function. 
Each instance of the unconstrained problem is then solved using Newton's method applied to a stationary-point equation, and this constitutes the inner loop of the IPM.
The connection of the stationary-point equation and the optimal solutions to the unconstrained problem is elucidated via the $\gamma$-stationary point, a kind of KKT point properly generalized to our nonconvex nonsmooth problem.
Furthermore, the numerical experiments have verified that the proposed IPM has an improved convergence rate compared to the BCD algorithm and the ADMM, as expected from second-order algorithms.

For future research, it will be interesting to extend the framework of the current paper to the dynamic version of factor analysis and other types of low-rank sparse sparse graphical models in the style of \cite{songsiri2010topology,ciccone2020learning,alpago2022scalable,falconi2024robust,you2024sparse}.
 

\bibliographystyle{IEEEtran}
\bibliography{refs}

\end{document}